\documentclass[a4paper,11pt]{article}
\usepackage[latin1]{inputenc}
\usepackage[english]{babel}
\usepackage{amsmath}
\usepackage{amsfonts}
\usepackage{amssymb}
\usepackage{epsfig}
\usepackage{amsopn}
\usepackage{amsthm}
\usepackage{color}
\usepackage{graphicx}
\usepackage{enumerate}
\usepackage{mathrsfs}
\usepackage{cite}
\newtheorem{theorem}{Theorem}[section]
\newtheorem{corollary}[theorem]{Corollary}
\newtheorem{lemma}[theorem]{Lemma}
\newtheorem{proposition}[theorem]{Proposition}

\newtheorem{remark}[theorem]{Remark}
\newtheorem*{theorem*}{Theorem}
\newtheorem*{lemma*}{Lemma}
\newtheorem*{remark*}{Remark}
\newtheorem*{definition*}{Definition}
\newtheorem*{proposition*}{Proposition}
\newtheorem*{corollary*}{Corollary}
\numberwithin{equation}{section}
\newcommand{\real}{\mathbb{R}}

\let\ced=\c         

\def\qed{\,\unskip\kern 6pt \penalty 500
\raise -2pt\hbox{\vrule \vbox to8pt{\hrule width 6pt
\vfill\hrule}\vrule}\par}
\definecolor{darkblue}{rgb}{0.05, .05, .65}
\definecolor{darkgreen}{rgb}{0.1, .65, .1}
\definecolor{darkred}{rgb}{0.8,0,0}
\newcommand{\beqn}{\begin{equation}}
\newcommand{\eeqn}{\end{equation}}
\newcommand{\bear}{\begin{eqnarray}}
\newcommand{\eear}{\end{eqnarray}}
\newcommand{\bean}{\begin{eqnarray*}}
\newcommand{\eean}{\end{eqnarray*}}
\begin{document}

\title{\huge \bf A Hardy-H\'enon equation in $\mathbb{R}^N$ with sublinear absorption}

\author{
\Large Razvan Gabriel Iagar\,\footnote{Departamento de Matem\'{a}tica
Aplicada, Ciencia e Ingenieria de los Materiales y Tecnologia
Electr\'onica, Universidad Rey Juan Carlos, M\'{o}stoles,
28933, Madrid, Spain, \textit{e-mail:} razvan.iagar@urjc.es}
\\[4pt] \Large Philippe Lauren\ced{c}ot\,\footnote{Laboratoire de Math\'ematiques (LAMA) UMR 5217, Universit\'e Savoie-Mont Blanc, CNRS, F-73000, Chamb\'ery France. \textit{e-mail:} philippe.laurencot@univ-smb.fr}\\ [4pt] }
\date{\today}
\maketitle

\begin{abstract}
Consider $m>1$, $N\ge 1$ and $\max\{-2,-N\}<\sigma<0$. The Hardy-H\'enon equation with sublinear absorption
$$
	- \Delta v(x) - |x|^\sigma v(x) + \frac{1}{m-1} v^{1/m}(x)= 0, \qquad x\in\mathbb{R}^N,
$$
is shown to have at least one solution $v\in H^1(\mathbb{R}^N)\cap L^{(m+1)/m}(\mathbb{R}^N)$, which is non-negative and radially symmetric with a non-increasing profile. In addition, any such solution is compactly supported, bounded and enjoys the better regularity $v\in W^{2,q}(\mathbb{R}^N)$ for $q\in [1,N/|\sigma|)$. A key ingredient in the proof is a particular case of the celebrated Caffarelli-Kohn-Nirenberg inequalities, for which we obtain the existence of an extremal function which is non-negative, bounded, compactly supported and radially symmetric with a non-increasing profile.

A by-product of these results is the existence of compactly supported separate variables solutions to a porous medium equation with a spatially dependent source featuring a singular coefficient.
\end{abstract}

\bigskip


\smallskip


\section{Introduction and results}

After being proposed by H\'enon \cite[Eq.~(A.6)]{He73} as a model for studying rotating stellar systems, the elliptic equation
\begin{equation}\label{eq0}
	-\Delta u(x)=|x|^{\sigma} u^p(x), \qquad x\in\mathbb{R}^N,
\end{equation}
with $p>1$, featuring a variable coefficient which is singular when $\sigma<0$ and unbounded when $\sigma>0$, became an interesting object of study for mathematicians. It is nowadays usually referred to as the Hardy equation if $\sigma<0$ and the H\'enon equation if $\sigma>0$, though both cases are considered in \cite{He73} (with a constraint corresponding to $\sigma>-2$). When $\sigma<0$, the presence of the singular potential implies a number of difficulties related to the existence and regularity of its solutions. Indeed, existence of solutions is established for Eq.~\eqref{eq0} with $N\ge 3$, $\sigma>-2$, and $p>p_S(\sigma)$ in works by Ni and his collaborators \cite{Ni82, Ni86, NY88}, where $p_S(\sigma):=(N+2\sigma+2)/(N-2)$ is the Sobolev critical exponent. In contrast, the existence issue is still not completely solved in the range $1<p<p_S(\sigma)$. New and rather complete results concerning existence of solutions are available in the recent paper \cite{GN22}, see also the references therein. Moreover, many new developments on the existence and the functional analysis of solutions to the parabolic counterpart of~\eqref{eq0}, namely
$$
	\partial_tu=\Delta u+|x|^{\sigma}u^p, \qquad (t,x)\in (0,\infty)\times\mathbb{R}^N,
$$
also referred to as the parabolic Hardy-H\'enon equation, are published in a number of recent works including \cite{BSTW17, CIT21a, CIT21b, CITT24, HT21, T20}.

Herein, we consider a Hardy-H\'enon equation featuring an additional sublinear absorption term
\begin{equation}\label{eq1}
-\Delta v(x)+\frac{1}{m-1}v^{1/m}(x)=|x|^{\sigma}v(x), \quad x\in\real^N,
\end{equation}
in dimension $N\geq1$ and with exponents
\begin{equation}\label{exp}
m>1 \quad {\rm and} \quad \max\{-2,-N\}<\sigma<0.
\end{equation}
The lower bound for $\sigma$ in~\eqref{exp} might seem strange at first sight, but even for the simpler equation~\eqref{eq0}, it has been shown that $\sigma>-2$ is a necessary condition for existence of solutions (see for example \cite{BC98, MP01}), while $\sigma>-N$ is the necessary condition for integrability in a neighborhood of $x=0$ of a non-negative and locally bounded function $v$ with $v(0)>0$. As seen from the previous discussion, the existence of solutions to Hardy equations is far from being a trivial subject, and we shall actually show in Theorem~\ref{th.3} below that the lower bound~\eqref{exp} on $\sigma$ is sharp in order for our results on~\eqref{eq1} to hold true when $N\ge 2$.

We are now in a position to state the main results of this paper, providing existence and regularity of solutions to~\eqref{eq1} when the parameters $m$ and $\sigma$ satisfy~\eqref{exp}. We start with the existence of at least one weak solution to~\eqref{eq1}, which turns out to be also an extremal function for a Caffarelli-Kohn-Nirenberg (CKN) inequality.

\begin{theorem}\label{th.1}
Let $m$ and $\sigma$ be as in~\eqref{exp}. We introduce
\begin{equation*}
	\mathcal{X} := H^1(\real^N)\cap L^{(m+1)/m}(\real^N) 
\end{equation*}
and
\begin{equation}
	\mathcal{S}(v) := \frac{\|\nabla v\|_2^{(\sigma+2)\theta - \sigma} \|v\|_{(m+1)/m}^{(\sigma+2)(1-\theta)}}{\|v\|_{L^2(\mathbb{R}^N,|x|^\sigma\,dx)}^2}, \qquad v\in \mathcal{X}\setminus\{0\}, \label{quotient}
\end{equation}
where
\begin{equation}
	\theta:=\frac{N(m-1)}{N(m-1)+2(m+1)}\in(0,1). \label{theta}
\end{equation}
There is at least one non-negative radially symmetric function $v_*\in \mathcal{X}$, $v_*\not\equiv 0$, with non-increasing profile, which minimizes $\mathcal{S}$ on $\mathcal{X}\setminus\{0\}$; that is,
\begin{equation}
	\mathcal{S}(v_*) = \inf_{v\in\mathcal{X}\setminus\{0\}} \mathcal{S}(v) > 0, \label{minimizer}
\end{equation}
and $v_*$ is a variational solution to~\eqref{eq1}.
\end{theorem}

Let us first point out that each term involved in the definition~\eqref{quotient} of $\mathcal{S}$ is finite for $v\in\mathcal{X}$, thanks to the continuous embedding of $\mathcal{X}$ in $L^2(\mathbb{R}^N,|x|^\sigma\,dx)$. As we shall see below in Lemma~\ref{lem.embed}, the latter is actually a consequence of a CKN inequality \cite[Theorem]{CKN84}, which also guarantees that the infimum of $\mathcal{S}$ on $\mathcal{X}\setminus\{0\}$ is positive. The original contribution of Theorem~\ref{th.1} is that this infimum is attained by a non-negative radially symmetric function with non-increasing profile solving~\eqref{eq1}. Let us recall that, in recent years, a number of works investigates the existence, properties, and stability of minimizers to CKN inequalities, see for example \cite{CC09, CFLL24, CW01, DE12, DETT11, DLL18, LL17, WW22} to name, but a few.

The proof of Theorem~\ref{th.1} involves two steps. The first one consists in showing the existence of a non-negative radially symmetric minimizer $v_*\in\mathcal{X}$, $v_*\not\equiv 0$, to an auxiliary constrained variational problem, combining the above mentioned CKN inequality with rearrangement arguments, see Section~\ref{subsec.exist}. Alongside, we show that $v_*$ is also a variational solution to~\eqref{eq1}. In a second step, we check that the minimizer $v_*$ thus obtained is also an extremal function for $\mathcal{S}$ on $\mathcal{X}\setminus\{0\}$, see Section~\ref{subsec.mp}.

\medskip

We next study in more details the properties of the extremal function $v_*$ constructed in Theorem~\ref{th.1}, which are actually also shared by any non-negative and radially symmetric variational solution $v\in \mathcal{X}$ to~\eqref{eq1} with non-increasing profile, $v\not\equiv 0$.

\begin{theorem}\label{th.2}
	Let $m$ and $\sigma$ be as in~\eqref{exp} and consider a non-negative and radially symmetric variational solution $v\in \mathcal{X}$ to~\eqref{eq1} with non-increasing profile, $v\not\equiv 0$. Then $v$ is compactly supported, bounded, and belongs to $W^{2,q}(\real^N)$ for any $q\in [1,N/|\sigma|)$. Moreover, introducing $r=|x|$ and $V(r) = V(|x|) = v(x)$ for $x\in\mathbb{R}^N$, there holds:
	\begin{itemize}
		\item If $\sigma\in (-1,0)$, then $V\in C^2((0,\infty))$ and, as $r\to 0$,
		\begin{equation}
			V''(r) = - \frac{\sigma+1}{\sigma+N} V(0) r^{\sigma} + \frac{V(0)^{1/m}}{N(m-1)} + o(1). \label{V0_1}
		\end{equation}
		\item If $\sigma = -1$, then $V\in C^2([0,\infty))$ and, as $r\to 0$,
		\begin{equation}
			V''(r) =  \frac{V(0)}{N(N-1)} + \frac{V(0)^{1/m}}{N(m-1)} + o(1). \label{V0_2}
		\end{equation}
		\item If $\sigma \in (-2,-1)$, then $V\in C^2((0,\infty))$ and, as $r\to 0$,
		\begin{equation}
			\begin{split}
			V''(r) & = - \frac{\sigma+1}{\sigma+N} V(0) r^{\sigma} + \frac{(2\sigma+3)}{(\sigma+2)(N+2\sigma+2)(\sigma+N)} V(0) r^{2(\sigma+1)} \\
			& \qquad + o\left( r^{2(\sigma+1)} \right).
			\end{split} \label{V0_3}
		\end{equation}
	\end{itemize}
\end{theorem}

A first consequence of Theorem~\ref{th.2} is that the extremal function $v_*$ of the CKN inequality~\eqref{minimizer} constructed in Theorem~\ref{th.1} is compactly supported, a feature which seems to have been unnoticed in the literature, as far as we know. The proof is provided in Section~\ref{subsec.compact} and applies to an arbitrary non-negative and radially symmetric variational solution $v\in \mathcal{X}$ to~\eqref{eq1} with non-increasing profile. It relies on the construction of a suitable supersolution and a comparison argument. We next study the Sobolev regularity and boundedness of $v$ in Section~\ref{subsec.br}. The cornerstone of the proof is the derivation of estimates on $v$ in $L^2(\mathbb{R}^N,|x|^\tau\,dx)$ for all $\tau\in (-N,0)$. As for the precise behavior of $v$ near $x=0$, we exploit the radial symmetry of $v$ and the corresponding ordinary differential equation. 

We further stress that the expansions~\eqref{V0_1} and~\eqref{V0_3} ensure that the stated regularity of $v$ is optimal for $\sigma\in (\max\{-N,-2\},0)$, except for $\sigma=-1$ (and $N\ge 2$). In fact, when $\sigma\in (\max\{-N,-2\},0)\setminus\{-1\}$, $v$ is not a classical solution to~\eqref{eq1}, but nevertheless belongs to $C^{1,\alpha}(\mathbb{R}^N)$ for any $\alpha\in (0,1+\sigma)$ when $\sigma \in (-1,0)$ and to $C^{\alpha}(\mathbb{R}^N)$ for any $\alpha\in (0,2+\sigma)$ when $\sigma \in (-2,-1)$. The case $\sigma=-1$ offers an interesting novelty and we infer from~\eqref{V0_2} that $v$ belongs at least to $C^2(\mathbb{R}^N)$ with 
\begin{equation*}
	D^2v(0) = \left( \frac{v(0)}{N(N-1)} + \frac{v(0)^{1/m}}{N(m-1)} \right) I_N,
\end{equation*}
where $I_N$ denotes the identity matrix of order $N$.

\medskip

As already mentioned, we supplement Theorem~\ref{th.1} with a non-existence result for~\eqref{eq1} when $\sigma\le -2$, thereby showing the optimality of the existence statement in Theorem~\ref{th.1} with respect to the parameter $\sigma$ when $N\ge 2$.

\begin{theorem}\label{th.3}
	Let $m>1$ and $\sigma\in (-\infty,-2]$. Then the equation~\eqref{eq1} has no non-trivial non-negative variational solution in $\mathcal{X}$.
\end{theorem}

The proof of Theorem~\ref{th.3} relies on the classical Pohozaev identity \cite{BL83, FT00} and is performed in Section~\ref{sec.ne}. It is unclear whether there exist non-trivial non-negative variational solutions $v\in\mathcal{X}$ to~\eqref{eq1} in space dimension $N=1$ in the remaining range $\sigma\in (-2,-1]$. It however follows from the continuous embedding of $H^1(\mathbb{R})$ in $L^\infty(\mathbb{R})$ that, if such a solution exists, then it has to vanish at $x=0$. In particular, it cannot be non-increasing as the solution provided by Theorem~\ref{th.1}.

\medskip

Besides its interest as an elliptic Hardy-H\'enon equation, another motivation for considering Eq.~\eqref{eq1} comes from the study of nonlinear diffusion equations with a source term which enjoys the same homogeneity. In particular, we are interested in describing the dynamics of the parabolic equation
\begin{equation}\label{eq2}
\partial_t u(t,x)=\Delta u^m(t,x)+|x|^{\sigma}u^m(t,x), \quad (t,x)\in(0,\infty)\times\real^N,
\end{equation}
with $N\geq1$ and $(m,\sigma)$ as in~\eqref{exp}. An essential first step in this direction is investigating the existence and properties of self-similar solutions, which are in separate variables form for~\eqref{eq2} and feature a finite time blow-up. More specifically,
\begin{equation}\label{SSS}
	u(t,x)=(T-t)^{-1/(m-1)}f(x), \quad (t,x)\in(0,T)\times\real^N, \ T\in(0,\infty).
\end{equation}
Inserting the ansatz~\eqref{SSS} into~\eqref{eq2} and letting $v=f^m$, we find after some straightforward calculations that $v$ solves the equation~\eqref{eq1}. We may then draw the following consequences of the previous analysis which apply to~\eqref{eq2}.

\begin{corollary}\label{cor.1}
Given $m$, $\sigma$ as in~\eqref{exp}, there exists at least one non-negative radially symmetric and compactly supported separate variables solution $u$ to~\eqref{eq2} in the form~\eqref{SSS} with a non-increasing profile $f$.
\end{corollary}

When $\sigma=0$, self-similar solutions to~\eqref{eq2} in separate variables form are constructed by a similar approach with the help of an auxiliary elliptic equation and their role in the dynamics of~\eqref{eq2} is elucidated in \cite{CdPE98, CdPE02, CEF96}, see also the references therein. The case $\sigma> 0$ in~\eqref{eq2} is considered recently by the authors and collaborators in \cite{IL22, IS20, IS22} where existence, non-existence, and classification of separate variables solutions to~\eqref{eq2} is obtained. In view of this discussion and precedents, it is strongly expected that the solutions to~\eqref{eq2} provided by Corollary~\ref{cor.1} play a significant role in the dynamics of~\eqref{eq2}.

\medskip

From now on, the parameters $m$ and $\sigma$ are assumed to satisfy~\eqref{exp} and we denote positive constants depending only on $N$, $m$, and $\sigma$ by $C$. Dependence upon additional parameters will be indicated explicitly. We also denote the measure of the unit ball $B(0,1)$ of $\mathbb{R}^N$ by $\omega_N$. 

\section{Existence and minimizing property}\label{sec.min}

The proof of Theorem~\ref{th.1} is based on a variational technique with constraints. We introduce the following functional
\begin{equation}\label{func}
J(v):=\frac{1}{2}\|\nabla v\|_{2}^2+\frac{m}{m^2-1}\|v\|_{(m+1)/m}^{(m+1)/m}, \quad v\in \mathcal{X},
\end{equation}
and the following weighted $L^2$-norm for $\tau\in\mathbb{R}$,
\begin{equation}\label{moment}
\mathcal{N}_{\tau}^2(v):=\int_{\real^N}|x|^{\tau}v^2(x)\,dx, \quad  v\in L^2(\real^N,|x|^{\tau}\,dx).
\end{equation}
The general idea of the proof is to minimize the functional $J$ in~\eqref{func} on the set
\begin{equation*}
\mathcal{A}:=\{v\in \mathcal{X}\ :\ \mathcal{N}_{\sigma}^2(v)=1\}.
\end{equation*}

\subsection{Existence and radial symmetry}\label{subsec.exist}

We first recall that, as a consequence of a CKN inequality \cite[Theorem]{CKN84}, the space $\mathcal{X}$ is continuously embedded in $L^2(\real^N,|x|^{\tau}\,dx)$ for any $\tau\in[-2,0)$.

\begin{lemma}\label{lem.embed}
Let $\tau\in[-2,0)$, $\tau>-N$. The space $\mathcal{X} = H^1(\real^N)\cap L^{(m+1)/m}(\real^N)$ embeds continuously in $L^2(\real^N,|x|^{\tau}\,dx)$ and there is $M(\tau)>0$ depending only on $N$, $m$, and $\tau$ such that
\begin{equation}\label{CKN}
	\mathcal{N}_{\tau}^2(w)\leq M(\tau) \|\nabla w\|_{2}^{(\tau+2)\theta-\tau}\|w\|_{(m+1)/m}^{(\tau+2)(1-\theta)}, \quad w\in \mathcal{X},
\end{equation}
the parameter $\theta$ being defined in~\eqref{theta}.
\end{lemma}

\begin{proof}
For $w\in C_{0}^{\infty}(\real^N)$, the functional inequality~\eqref{CKN} is a consequence of \cite[Theorem]{CKN84} by letting, in the notation therein, $p=2$, $r=2$, $q=(m+1)/m$, $\alpha=\beta=0$, $\gamma=\tau/2$, and $\sigma=\tau/[(\tau+2)\theta-\tau]<0$, with the exponent $a$ therein given by
$$
	\frac{1}{2}+\frac{\tau}{2N}=a\left(\frac{1}{2}-\frac{1}{N}\right)+(1-a)\frac{m}{m+1},
$$
whence
$$
	a=\frac{N(m-1)-\tau(m+1)}{N(m-1)+2(m+1)}=\frac{(\tau+2)\theta-\tau}{2},
$$
as claimed. Then, the density of $C_{0}^{\infty}(\real^N)$ in $\mathcal{X}$ extends the result to general functions $w\in\mathcal{X}$.
\end{proof}

We next show that we can restrict the analysis of $J$ to non-negative radially symmetric functions with non-increasing profiles.

\begin{lemma}\label{lem.radial}
Let $v\in\mathcal{A}$. There exists $w\in\mathcal{A}_{SD}^{+}$ such that $J(v)\geq J(w)$, where $\mathcal{A}_{SD}^{+}$ is the subset of $\mathcal{A}$ consisting only of non-negative radially symmetric functions with non-increasing profiles. In particular,
\begin{equation*}
	\inf_{v\in \mathcal{A}} J(v) = \inf_{w\in\mathcal{A}_{SD}^{+}} J(w).
\end{equation*}
\end{lemma}

\begin{proof}
Consider $v\in\mathcal{A}$ and let $|v|^*$ be the symmetric-decreasing rearrangement of $|v|$. On the one hand, Polya's inequality \cite[7.17~Lemma]{LL01} ensures that
\begin{equation}\label{interm1}
	\|\nabla v\|_2=\|\nabla|v|\|_2\geq\|\nabla|v|^*\|_2,
\end{equation}
while the equimeasurability property of the rearrangement gives
\begin{equation}\label{interm2}
	\|v\|_{(m+1)/m}=\||v|^*\|_{(m+1)/m}.
\end{equation}
On the other hand, we infer from the pointwise identity $(v^2)^*=(|v|^*)^2$ \cite[Eq.~(3.8)]{Li77}, the monotonicity of $x\mapsto |x|^{\sigma}$ and the Hardy-Litlewood inequality \cite[3.4~Theorem]{LL01} that
\begin{align*}
	\mathcal{N}_{\sigma}^2(v) & = \int_{\real^N} |x|^{\sigma} v^2(x)\,dx \leq \int_{\real^N} |x|^{\sigma} (v^2)^*(x)\,dx\\
	& = \int_{\real^N} |x|^{\sigma} (|v|^*)^2(x)\,dx = \mathcal{N}_{\sigma}^2(|v|^*),
\end{align*}
from which we deduce that, since $v\in\mathcal{A}$,
\begin{equation}\label{interm3}
	\mathcal{N}_{\sigma}^2(|v|^*)\geq1.
\end{equation}
Introducing $w:=|v|^*/\mathcal{N}_{\sigma}(|v|^*)$, it is then obvious that $w\in\mathcal{A}_{SD}^{+}$ and we obtain from~\eqref{interm1}, \eqref{interm2} and~\eqref{interm3} that
\begin{align*}
	J(w) & = \frac{1}{2} \frac{\|\nabla |v|^*\|_2^{2}}{\mathcal{N}_{\sigma}^2(|v|^*)} + \frac{m}{m^2-1} \frac{\||v|^*\|_{(m+1)/m}^{(m+1)/m}}{\mathcal{N}_{\sigma}^{(m+1)/m}(|v|^{*})}\\
	&\leq \frac{1}{2} \|\nabla |v|^*\|_2^{2} + \frac{m}{m^2-1}\||v|^*\|_{(m+1)/m}^{(m+1)/m}\\
	&\leq \frac{1}{2} \|\nabla v\|_2^{2} + \frac{m}{m^2-1}\|v\|_{(m+1)/m}^{(m+1)/m}=J(v),
\end{align*}
as claimed.
\end{proof}

We are now in a position to prove the existence of a minimizer of $J$ on $\mathcal{A}$.

\begin{proposition}\label{prop.existSD}
There is $v_*\in \mathcal{A}_{SD}^{+}$ such that
\begin{equation*}
	J(v_*) = \inf_{w\in\mathcal{A}} J(v).
\end{equation*}
In addition, $v_*$ is a variational solution to~\eqref{eq1}.
\end{proposition}

\begin{proof}
Since $J$ is non-negative, it follows from Lemma~\ref{lem.radial} that
\begin{equation*}
	\nu =  \inf_{v\in\mathcal{A}} J(v) = \inf_{w\in\mathcal{A}_{SD}^{+}} J(w)
\end{equation*}
is well-defined and non-negative. There is thus a sequence $(w_j)_{j\geq1}$ such that $w_j\in\mathcal{A}_{SD}^{+}$ and
\begin{equation}
	\nu\leq J(w_j)\leq\nu+\frac{1}{j}, \quad j\geq1. \label{add05}
\end{equation}
It readily follows from~\eqref{add05} that $(w_j)_{j\geq1}$ is a bounded sequence in $\mathcal{X}$. By a classical compactness result (see \cite[Appendix~A.I]{BL83} for instance), there are a subsequence of $(w_j)_{j\geq1}$ (not relabeled) and a function $w\in \mathcal{X}$ such that
\begin{subequations}\label{conv}
\begin{equation}\label{conv.grad}
\nabla w_j\rightharpoonup\nabla w \quad {\rm in} \ L^2(\real^N),
\end{equation}
\begin{equation}\label{conv.Lm}
w_j\rightharpoonup w \quad {\rm in} \ L^{(m+1)/m}(\real^N),
\end{equation}
\begin{equation}\label{conv.L2}
w_j\to w \quad {\rm in} \ L^2(\real^N).
\end{equation}
\end{subequations}
A first consequence of Lemma~\ref{lem.embed} and the convergences in~\eqref{conv} is that $w\in L^2(\real^N,|x|^{\tau}\,dx)$ for any $\tau\in[-2,0)$, $\tau>-N$, and in particular for $\tau=\sigma$. Moreover,
\begin{equation}\label{interm4}
\|\nabla w\|_{2}^{2}\leq \liminf\limits_{j\to\infty}\|\nabla w_j\|_{2}^2, \quad \|w\|_{(m+1)/m}^{(m+1)/m}\leq\liminf\limits_{j\to\infty}\|w_j\|_{(m+1)/m}^{(m+1)/m}
\end{equation}
and we deduce from~\eqref{add05} and~\eqref{interm4} that
\begin{equation}\label{interm5}
J(w)\leq \liminf\limits_{j\to\infty}J(w_j)=\nu.
\end{equation}
The next step is to prove that $\mathcal{N}_{\sigma}(w)=1$. To this end, we fix $\tau\in (\max\{-N,-2\},\sigma)$ and observe that $\mathcal{X}$ embeds continuously in $L^2(\real^N,|x|^{\tau}\,dx)$ according to Lemma~\ref{lem.embed}. Therefore, for $\epsilon\in (0,1)$,
\begin{align*}
	\mathcal{N}_\sigma^2(w_j-w) & = \int_{B(0,\epsilon)} |x|^\sigma |(w_j-w)(x)|^2\,dx + \int_{\mathbb{R}^N\setminus B(0,\epsilon)} |x|^\sigma |(w_j-w)(x)|^2\,dx \\
	& \le \epsilon^{\sigma-\tau} \int_{B(0,\epsilon)} |x|^\tau |(w_j-w)(x)|^2\,dx + \epsilon^\sigma \int_{\mathbb{R}^N\setminus B(0,\epsilon)} |(w_j-w)(x)|^2\,dx \\
	& \le \epsilon^{\sigma-\tau} \mathcal{N}_\tau^2(w_j-w) + \epsilon^\sigma \|w_j-w\|_2^2 \\
	& \le C(\tau) \epsilon^{\sigma-\tau} \left( \|w\|_{\mathcal{X}}^2 + \sup_{j\ge 1} \|w_j\|_{\mathcal{X}}^2 \right) + \epsilon^\sigma \|w_j-w\|_2^2.
\end{align*}
We then let $j\to\infty$ and deduce from~\eqref{conv.L2} that
\begin{equation*}
	\limsup\limits_{j\to\infty} \mathcal{N}_\sigma^2(w_j-w) \le C(\tau) \epsilon^{\sigma-\tau} \left( \|w\|_{\mathcal{X}}^2 + \sup_{j\ge 1} \|w_j\|_{\mathcal{X}}^2 \right).
\end{equation*}
Recalling that $\sigma-\tau>0$, we then let $\epsilon\to 0$ and end up with
$$
	\lim\limits_{j\to\infty} \mathcal{N}_{\sigma}^2(w_j-w)=0.
$$
Since $w_j\in\mathcal{A}_{SD}^{+}$, we deduce from the previous convergence that $\mathcal{N}_{\sigma}^2(w)=1$. Also, the convergence~\eqref{conv.L2} implies the almost everywhere convergence of a subsequence of $(w_j)_{j\geq1}$ to $w$, from which the radial symmetry and monotonicity of $w$ follow. Therefore, we have just established that $w\in\mathcal{A}_{SD}^+$, while~\eqref{interm5} and the definition of $\nu$ ensure that $J(w)=\nu$. The minimizing property of $w$ then yields that, for any $\varphi\in C_0^\infty(\mathbb{R}^N)$,
\begin{equation}\label{interm6}
	\int_{\mathbb{R}^N} \left[ \nabla w(x)\cdot \nabla\varphi(x) + \frac{1}{m-1} w^{1/m}(x) \varphi(x) - \lambda |x|^{\sigma} w(x) \varphi(x) \right]\,dx =0
\end{equation}
with
\begin{equation*}
	\lambda := \|\nabla w\|_{2}^{2}+\frac{1}{m-1}\|w\|_{(m+1)/m}^{(m+1)/m}>0.
\end{equation*}
We next set
$$
v(x):=\lambda^{2m/[(m-1)(\sigma+2)]}w(\lambda^{-1/(\sigma+2)}x), \quad x\in\real^N.
$$
It is obvious that $v$ is a non-negative radially symmetric function with non-increasing profile and that $v\in \mathcal{X}$. Moreover, for $\varphi\in C_0^\infty(\mathbb{R}^N)$, the function $\psi: x\mapsto \varphi(\lambda^{1/(\sigma+2)}x)$ also belongs to $C_0^\infty(\mathbb{R}^N)$ and we have
\begin{align*}
	\int_{\mathbb{R}^N} \nabla w(x)\cdot \nabla\psi(x)\, dx & = \lambda^{-[2+N(m-1)]/[(m-1)(\sigma+2)]} \int_{\mathbb{R}^N} \nabla v(x)\cdot \nabla\varphi(x)\, dx, \\
	\int_{\mathbb{R}^N} w^{1/m}(x) \psi(x)\,dx & = \lambda^{-[2+N(m-1)]/[(m-1)(\sigma+2)]} \int_{\mathbb{R}^N} v^{1/m}(x) \varphi(x)\, dx, \\
	\lambda \int_{\mathbb{R}^N} |x|^{\sigma} w(x) \psi(x)\,dx & = \lambda^{-[2+N(m-1)]/[(m-1)(\sigma+2)]} \int_{\mathbb{R}^N} |x|^\sigma v(x) \varphi(x)\,dx.
\end{align*}
Inserting these identities in~\eqref{interm6} (with $\psi$ instead of $\varphi$) gives
\begin{equation}\label{add06}
	\int_{\mathbb{R}^N} \left[ \nabla v(x)\cdot \nabla\varphi(x) + \frac{1}{m-1} v^{1/m}(x) \varphi(x) - |x|^{\sigma} v(x) \varphi(x) \right]\,dx =0,
\end{equation}
so that $v$ is a weak solution to~\eqref{eq1}.

Let us finally prove that $v$ is a variational solution to~\eqref{eq1}; that is, the weak formulation~\eqref{add06} is satisfied for all $\varphi\in H^1(\mathbb{R}^N)$. Indeed, if $N\ge 3$ and $\varphi\in H^1(\mathbb{R}^N)$, then we infer from H\"older's inequality that
\begin{equation*}
	\left| \int_{\mathbb{R}^N} |x|^\sigma v(x) \varphi(x)\,dx \right| \le \int_{\mathbb{R}^N} |x|^{\sigma/2} v(x) |x|^{\sigma/2} |\varphi(x)|\,dx\le \mathcal{N}_\sigma(v) \mathcal{N}_\sigma(\varphi),
\end{equation*}
while the Sobolev inequality
\begin{equation}
	\|\varphi\|_{2^*} \le C_S \|\nabla\varphi\|_{2}, \qquad \varphi\in H^1(\mathbb{R}^N), \qquad 2^* := \frac{2N}{N-2}>2, \label{Isobolev}
\end{equation}
gives
\begin{align*}
	\mathcal{N}_\sigma^2(\varphi) & = \int_{B(0,1)} |x|^\sigma \varphi^2(x)\, dx + \int_{\mathbb{R}^N\setminus B(0,1)} |x|^\sigma \varphi^2(x)\, dx \\
	& \le \left( \int_{B(0,1)} |\varphi(x)|^{2^*}\,dx \right)^{2/2^*} \left( \int_{B(0,1)} |x|^{\sigma N/2}\,dx \right)^{2/N} + \int_{\mathbb{R}^N\setminus B(0,1)} \varphi^2(x)\, dx \\
	& \le C_S^2 \left( \frac{2\omega_N }{\sigma+2} \right)^{2/N} \|\nabla\varphi\|_{2}^2 + \|\varphi\|_2^2.
\end{align*}
Therefore, thanks to the subadditivity of the square root,
\begin{equation*}
	\left| \int_{\mathbb{R}^N} |x|^\sigma v(x) \varphi(x)\,dx \right| \le \mathcal{N}_\sigma(v) \left[ 1 + C_S \left( \frac{2\omega_N }{\sigma+2} \right)^{1/N} \right] \|\varphi\|_{H^1},
\end{equation*}
and a duality argument implies that $x\mapsto |x|^\sigma v(x)\in H^{-1}(\mathbb{R}^N)$. Since $\Delta v$ also belongs to $H^{-1}(\mathbb{R}^N)$ and $v^{1/m}(x) = (m-1) [\Delta v(x) + |x|^\sigma v(x)]$ in $\mathcal{D}'(\mathbb{R}^N)$, we conclude that $v^{1/m}\in H^{-1}(\mathbb{R}^N)$. Since the three terms involved in~\eqref{add06} belong to $H^{-1}(\mathbb{R}^N)$, a density argument allows us to extend the validity of~\eqref{add06} to any $\varphi\in H^1(\mathbb{R}^N)$.

For $N\in\{1,2\}$, we employ similar arguments and provide a sketch of the needed modifications. For $N=2$, we fix $p\in (1,2/|\sigma|)$ and use H\"older's inequality to obtain that
\begin{align*}
	\int_{B(0,1)} |x|^\sigma \varphi^2(x)\, dx & \le \left( \int_{B(0,1)} \varphi(x)^{2p/(p-1)}\, dx \right)^{(p-1)/p} \left( \int_{B(0,1)} |x|^{\sigma p}\, dx \right)^{1/p} \\
	& \le C(p) \left( \frac{2\pi}{2+\sigma p} \right)^{1/p} \|\varphi\|_{H^1}^2,
\end{align*}
thanks to the continuous embedding of $H^1(\mathbb{R}^2)$ in $L^{2p/(p-1)}(\mathbb{R}^2)$, and we proceed as above. For $N=1$, we just use the continuous embedding of $H^1(\mathbb{R})$ in $L^{\infty}(\mathbb{R})$
and the integrability of $x\mapsto |x|^\sigma$ on $B(0,1)=(-1,1)$ due to $\sigma\in (-1,0)$.
\end{proof}

\subsection{Minimizing property}\label{subsec.mp}

We next complete the proof of Theorem~\ref{th.1} by proving below that the minimizer $v_*$ constructed in Proposition~\ref{prop.existSD} is an extremal function to the CKN inequality~\eqref{minimizer}. To this end, we employ a classical technique based on a scaling argument. To simplify the notation, we set $\mu:=(m+1)/m$.

\begin{proof}[Proof of Theorem~\ref{th.1}: minimizing property]
For $w\in\mathcal{X}$ and $\lambda>0$, we set
\begin{equation*}
	w_\lambda(x) := \lambda w\left( \lambda^{2/(N+\sigma)} \mathcal{N}_\sigma(w)^{2/(N+\sigma)} x \right), \quad x\in\mathbb{R}^N.
\end{equation*}
Then we find by straightforward calculations that
\begin{align*}
	\|\nabla w_\lambda\|_2^2 & = \lambda^{2(\sigma+2)/(N+\sigma)} \mathcal{N}_\sigma(w)^{-2(N-2)/(N+\sigma)} \|\nabla w\|_2^2, \\
	\|w_\lambda\|_{\mu}^{\mu} & = \lambda^{[\sigma(m+1)-N(m-1)]/[m(N+\sigma)]} \mathcal{N}_\sigma(w)^{-2N/(N+\sigma)} \|w\|_{\mu}^{\mu}, \\
	\mathcal{N}_\sigma^2(w_\lambda) & := \mathcal{N}_\sigma(w)^{-2} \mathcal{N}_\sigma(w)^2 = 1.
\end{align*}	
Consequently, $w_\lambda\in\mathcal{A}$ and, since $J(v_*) \le J(w_\lambda)$ by  Proposition~\ref{prop.existSD}, we infer from the above formulas that
\begin{equation}
\mathcal{N}_\sigma(w)^{2N/(N+\sigma)}\le \mathcal{J}(\lambda,w) := \mathcal{N}_\sigma(w)^{2N/(N+\sigma)} \frac{J(w_\lambda)}{J(v_*)} , \label{add07}
\end{equation}
with
\begin{align*}
	\mathcal{J}(\lambda,w) & := A(w) \lambda^a + B(w) \lambda^{-b}, \\
	A(w) & := \frac{1}{2} \mathcal{N}_\sigma(w)^{4/(N+\sigma)} \frac{\|\nabla w\|_2^2}{J(v_*)}, \qquad a := \frac{2(\sigma+2)}{N+\sigma} >0 , \\
	B(w) & := \frac{m}{(m^2-1)} \frac{\|w\|_\mu^\mu}{J(v_*)}, \qquad b:= \frac{N(m-1)-\sigma(m+1)}{m(N+\sigma)}> 0.
\end{align*}
Since~\eqref{add07} is valid for any $\lambda>0$, by optimizing $J(\lambda,w)$ with respect to $\lambda$, we end up with
\begin{equation}
		\mathcal{N}_\sigma(w)^{2N/(N+\sigma)} \le j(w) := \mathcal{J}\left( \left[ \frac{b B(w)}{a A(w)} \right]^{1/(a+b)} , w\right) = \inf_{\lambda>0} \mathcal{J}(\lambda,w). \label{add08}
\end{equation}
On the one hand, it follows from the properties of $v_*$ that $A(v_*)+B(v_*)=1$, whence
\begin{equation}
	1 = \mathcal{N}_\sigma(v_*)^{2N/(N+\sigma)} = \mathcal{J}(1,v_*) = j(v_*). \label{add09}
\end{equation}
On the other hand, we infer from~\eqref{add08} that, for $w\in\mathcal{X}$,
\begin{align*}
	\mathcal{N}_\sigma(w)^{2N/(N+\sigma)} & \le j(w) = \mathcal{J}\left( \left[ \frac{b B(w)}{a A(w)} \right]^{1/(a+b)} , w\right) \\
	& = (a+b) a^{-a/(a+b)} b^{-b/(a+b)} A(w)^{b/(a+b)} B(w)^{a/(a+b)},
\end{align*}
or, equivalently,
\begin{equation}
	\mathcal{N}_\sigma(w)^{2N(a+b)/(N+\sigma)} \le j(w)^{a+b} = \frac{(a+b)^{a+b}}{a^a b^b} A(w)^b B(w)^a. \label{add10}
\end{equation}
Setting
\begin{equation*}
	J_* := \left( \frac{J(v_*)}{a+b} \right)^{a+b} \left( \frac{a(m^2-1)}{m} \right)^a (2b)^b,
\end{equation*}
we replace $A(w)$ and $B(w)$ by their formulas in~\eqref{add10} to find
\begin{equation*}
	\mathcal{N}_\sigma(w)^{[2N(a+b)-4b]/(N+\sigma)} \le \frac{j(w)^{a+b}}{\mathcal{N}_\sigma(w)^{4b/(N+\sigma)}} = \frac{\|\nabla w\|_2^{2b} \|w\|_\mu^{\mu a}}{J_*}.
\end{equation*}
Since
\begin{align*}
	\frac{2N(a+b)-4b}{N+\sigma} & = 2 \frac{N(m-1)+2(m+1)}{m(N+\sigma)}, \\
	2b & = \big[ \theta(\sigma+2)-\sigma \big] \frac{N(m-1)+2(m+1)}{m(N+\sigma)}, \\
	\mu a & = (\sigma+2)(1-\theta) \frac{N(m-1)+2(m+1)}{m(N+\sigma)},
\end{align*}
the parameter $\theta\in (0,1)$ being defined in~\eqref{theta}, we further obtain
\begin{equation}
		\mathcal{N}_\sigma^2(w) \le \left[ \frac{j(w)^{a+b}}{\mathcal{N}_\sigma(w)^{4b/(N+\sigma)}} \right]^{m(N+\sigma)/[N(m-1)+2(m+1)]} = \frac{\mathcal{N}_\sigma^2(w)\mathcal{S}(w)}{K_*}, \label{add11}
\end{equation}
with 
\begin{equation*}
	K_* := J_*^{m(N+\sigma)/[N(m-1)+2(m+1)]},
\end{equation*} 
recalling that $\mathcal{S}$ is defined by~\eqref{quotient}. We then readily infer from~\eqref{add11} that $\mathcal{S}(w)\ge K_*$ for all $w\in\mathcal{X}$, $w\not\equiv 0$, while the specific choice $w=v_*$ in~\eqref{add11} yields $\mathcal{S}(v_*)=K_*$ due to $v_*\in\mathcal{A}$ and~\eqref{add09}.
\end{proof}

\section{Qualitative properties}\label{sec.qp}

The aim of this section is to prove Theorem \ref{th.2}. For the reader's convenience, its proof is split into a number of steps contained in separate subsections. 

\subsection{Compactness of the support}\label{subsec.compact}

As already announced, compactness of the support of non-negative radially symmetric solutions to~\eqref{eq1} with non-increasing profiles is proved by a comparison argument. To this end, the following result will be very useful.

\begin{lemma}\label{lem.tail}
Let $v\in \mathcal{X}$ be a non-negative radially symmetric solution to~\eqref{eq1} with non-increasing profile. Then there is $r_0>0$ such that
$$
-\Delta v(x)+\frac{1}{2(m-1)}v^{1/m}(x)\leq 0, \quad x\in\real^N\setminus B(0,r_0).
$$
\end{lemma}

\begin{proof}
Given $\varrho>0$ and taking into account the symmetry and monotonicity of $v$, we have for $x\in\real^N\setminus B(0,\varrho)$
\begin{equation}\label{interm7}
|x|^{\sigma}v(x)=|x|^{\sigma}v^{1/m}(x)v^{(m-1)/m}(x)\leq \varrho^{\sigma} V(\varrho)^{(m-1)/m}v(x)^{1/m},
\end{equation}
where $V(|x|)=v(x)$, $x\in\mathbb{R}^N$. Since $V$ is non-increasing and $\sigma<0$, the application $\varrho\mapsto \varrho^{\sigma}V(\varrho)^{(m-1)/m}$ is decreasing and tends to zero as $\varrho\to\infty$. Thus, we can pick $r_0>0$ such that
\begin{equation*}
r_0^{\sigma}V(r_0)^{(m-1)/m}\leq\frac{1}{2(m-1)}.
\end{equation*}
Consequently, \eqref{eq1} and~\eqref{interm7} (with $\varrho=r_0$) give
$$
-\Delta v(x)+\frac{1}{m-1}v^{1/m}(x)=|x|^{\sigma}v(x)\leq\frac{1}{2(m-1)}v^{1/m}(x), \quad x\in\real^N\setminus B(0,r_0),
$$
from which the claim follows.
\end{proof}

The next technical step is the construction of a family of supersolutions that will be employed for comparison later on.

\begin{lemma}\label{lem.super}
Let $(a,b,c)\in (0,\infty)^3$. Then
$$
V_{a,b}(x) := \big(a-b|x|^2\big)_{+}^{2m/(m-1)}, \quad x\in\real^N,
$$
is a supersolution to the equation
\begin{equation*}
-\Delta v(x)+cv^{1/m}(x)=0, \quad x\in\real^N,
\end{equation*}
provided $ab\leq c(m-1)^2/8m(m+1)$.
\end{lemma}

\begin{proof}
Set $\omega:=2m/(m-1)>2$. We then compute, for $x\in\real^N$,
$$
\nabla V_{a,b}(x)=-2b\omega(a-b|x|^2)_{+}^{\omega-1}x
$$
and
\begin{equation*}
\begin{split}
\Delta V_{a,b}(x)&=-2Nb\omega(a-b|x|^2)_{+}^{\omega-1}+4b^2\omega(\omega-1)|x|^2(a-b|x|^2)_{+}^{\omega-2}\\
&=-2b\omega[N+2(\omega-1)](a-b|x|^2)_{+}^{\omega-1}+4ab\omega(\omega-1)(a-b|x|^2)_{+}^{\omega-2}.
\end{split}
\end{equation*}
Hence,
\begin{equation*}
\begin{split}
-\Delta V_{a,b}(x)+cV_{a,b}^{1/m}(x)&= 2b\omega[N+2(\omega-1)](a-b|x|^2)_{+}^{\omega-1}\\
&-4ab\omega(\omega-1)(a-b|x|^2)_{+}^{\omega-2}+c(a-b|x|^2)_{+}^{\omega/m}.
\end{split}
\end{equation*}
Since $\omega>2$ and $\omega/m=\omega-2$, we discard the first term and we further infer that
$$
-\Delta V_{a,b}(x)+cV_{a,b}^{1/m}(x)\geq[c-4ab\omega(\omega-1)](a-b|x|^2)_{+}^{\omega-2}
$$
and the right hand side of the above inequality is non-negative, provided
$$
ab\leq\frac{c}{4\omega(\omega-1)}=\frac{c(m-1)^2}{8m(m+1)},
$$
as stated.
\end{proof}

We are now in a position to prove that the non-negative radially symmetric variational solutions to~\eqref{eq1} with non-increasing profiles have compact support.

\begin{proposition}\label{prop.compact}
Let $v\in H^1(\real^N)\cap L^{(m+1)/m}(\real^N)$ be a non-negative radially symmetric solution to~\eqref{eq1} with non-increasing profile. Then $v$ is compactly supported.
\end{proposition}

\begin{proof}
By Lemma~\ref{lem.tail}, there is $r_0>0$ such that
\begin{equation}\label{interm8}
-\Delta v(x)+\frac{1}{2(m-1)}v^{1/m}(x)\leq 0, \quad x\in\real^N\setminus B(0,r_0).
\end{equation}
Since $v\in L^{(m+1)/m}(\mathbb{R}^N)$, we infer from \cite[Radial Lemma~A.IV]{BL83} that
\begin{equation*}
|v(x)|\leq C\|v\|_{(m+1)/m}|x|^{-mN/(m+1)}, \quad x\in\real^N\setminus\{0\},
\end{equation*}
hence $M_0:=\sup\{v(x): x\in\partial B(0,r_0)\}<\infty$. It next follows from Lemma~\ref{lem.super} with $c=1/2(m-1)$ that
\begin{equation}\label{interm9}
-\Delta V_{a,b}(x)+\frac{1}{2(m-1)}V_{a,b}^{1/m}(x)\geq0, \quad x\in\real^N,
\end{equation}
provided $ab\leq (m-1)/4m(m+1)$. Choosing $a$ and $b$ in order to satisfy
$$
a=br_0^2+M_0^{(m-1)/2m}, \quad b(br_0^2+M_0^{(m-1)/2m})\leq\frac{m-1}{4m(m+1)},
$$
we also observe that
$$
V_{a,b}(x)=(a-br_0^2)^{2m/(m-1)} = M_0 \geq v(x), \quad x\in\partial B(0,r_0).
$$
Consider now $\varphi\in C_0^{\infty}(\real^N)$ such that $0\leq\varphi\leq1$, $\varphi=1$ on $B(0,1)$ and $\varphi=0$ on $\real^N\setminus B(0,2)$ and set $\varphi_n(x):=\varphi(x/n)$ for $x\in\real^N$ and $n\geq1$. Introducing $W:=v-V_{a,b}$, we infer from~\eqref{interm8} and~\eqref{interm9} that, for $n\geq1$,
$$
-\varphi_nW_{+}\Delta W\leq \varphi_nW_+\left[-\frac{1}{2(m-1)}v^{1/m}+\frac{1}{2(m-1)}V_{a,b}^{1/m}\right] \quad {\rm in} \ \real^N\setminus B(0,r_0).
$$
We then obtain by integration on $\real^N\setminus B(0,r_0)$ that
$$
-\int_{\real^N\setminus B(0,r_0)}\varphi_nW_{+}\Delta W\,dx+\int_{\real^N\setminus B(0,r_0)}\varphi_nW_{+}\frac{v^{1/m}-V_{a,b}^{1/m}}{2(m-1)}\,dx\leq0.
$$
We integrate by parts in the first integral of the previous estimate, taking into account that $W_+=0$ on $\partial B(0,r_0)$ and 
\begin{equation*}
	\nabla W_+\cdot \nabla W = |\nabla W_+|^2 \;\text{ and }\; W_+\nabla W = \frac{1}{2} \nabla W_+^2 \;\;\text{ a.e. in }\;\; \mathbb{R}^N
\end{equation*}
to obtain that
\begin{equation}\label{interm10}
\frac{1}{2}\int_{\real^N\setminus B(0,r_0)}\nabla W_{+}^2\cdot\nabla\varphi_n\,dx+\frac{1}{2(m-1)}\int_{\real^N\setminus B(0,r_0)}\varphi_nW_+(v^{1/m}-V_{a,b}^{1/m})\,dx\leq0.
\end{equation}
Using again that $W_+=0$ on $\partial B(0,r_0)$ and integrating by parts in the first term of~\eqref{interm10} lead to
\begin{equation*}
\begin{split}
\frac{1}{m-1}\int_{\real^N\setminus B(0,r_0)}\varphi_nW_+(v^{1/m}-V_{a,b}^{1/m})\,dx&\leq\int_{\real^N\setminus B(0,r_0)}W_{+}^2\Delta\varphi_n\,dx\\&\leq\frac{1}{n^2}\int_{\real^N\setminus B(0,r_0)}W_+^2\Delta\varphi\left(\frac{x}{n}\right)\,dx\\
&\leq\frac{\|\Delta\varphi\|_{\infty}}{n^2}\int_{\real^N\setminus B(0,r_0)}W_{+}^2\,dx.
\end{split}
\end{equation*}
On the one hand, we recall that $\mathcal{X}$ is continuously embedded in $L^2(\real^N)$ since $(m+1)/m<2<2^*$, so that $v\in L^2(\real^N)$. Thus, $W_{+}\in L^2(\real^N\setminus B(0,r_0))$ and we find
\begin{equation}\label{interm11}
\lim\limits_{n\to\infty}\frac{\|\Delta\varphi\|_{\infty}}{n^2}\int_{\real^N\setminus B(0,r_0)}W_{+}^2\,dx=0.
\end{equation}
On the other hand, we recall that $\varphi_n\geq0$ in $\real^N$, $\lim\limits_{n\to\infty}\varphi_n(x)=1$ pointwisely in $\real^N$ and
$$
W_{+}(v^{1/m}-V_{a,b}^{1/m})\geq0 \quad {\rm in} \ \real^N\setminus B(0,r_0).
$$
We thus readily infer from Fatou's lemma and~\eqref{interm11} that
$$
0\leq\int_{\real^N\setminus B(0,r_0)}W_{+}(v^{1/m}-V_{a,b}^{1/m})\,dx\leq\liminf\limits_{n\to\infty}\int_{\real^N\setminus B(0,r_0)}\varphi_nW_{+}(v^{1/m}-V_{a,b}^{1/m})\,dx\leq0,
$$
whence
$$
\int_{\real^N\setminus B(0,r_0)}W_{+}(v^{1/m}-V_{a,b}^{1/m})\,dx=0.
$$
The latter implies that $v(x)\leq V_{a,b}(x)$ for $x\in\real^N\setminus B(0,r_0)$, from which we deduce that $v$ has compact support and the proof is complete.
\end{proof}

\subsection{Boundedness and regularity}\label{subsec.br}

Let $v\in\mathcal{X}$ be a non-negative radially symmetric solution to~\eqref{eq1} with non-increasing profile. The radial symmetry, monotonicity, and $L^{(m+1)/m}$-integrability of $v$ guarantee that $v$ belongs to $L^\infty\big(\mathbb{R}^N\setminus B(0,R)\big)$ for any $R>0$, but we do not know yet whether $v(0)$ is finite, or equivalently, $v\in L^\infty(\mathbb{R}^N)$. At this stage, the boundedness of $v$ is only known when $N=1$, as a consequence of the continuous embedding of $H^1(\mathbb{R})$ in $L^\infty(\mathbb{R})$. To handle higher space dimensions, a refined study of the behavior of $v$ near $x=0$ is needed and the next lemma is a first step in that direction.

\begin{lemma}\label{lem.moments}
Let $N\geq 3$. For any $\tau\in (-N,0)$, we have $\mathcal{N}_{\tau}(v)<\infty$ and $\mathcal{N}_{\tau+2}(|\nabla v|)<\infty$.
\end{lemma}

\begin{proof}
Let $k>0$ and $\epsilon\in(0,1)$. We deduce from~\eqref{eq1} by multiplying by $v(x)/(|x|^2+\epsilon)^k$ and integrating, that
\begin{equation*}
\begin{split}
\int_{\real^N}\frac{|x|^{\sigma}v(x)^2}{(|x|^2+\epsilon)^k}\,dx&-\frac{1}{m-1}\int_{\real^N}\frac{v^{(m+1)/m}(x)}{(|x|^2+\epsilon)^k}\,dx
=\int_{\real^N}\nabla\left(\frac{v(x)}{(|x|^2+\epsilon)^k}\right)\cdot\nabla v(x)\,dx\\
&=\int_{\real^N}\frac{|\nabla v(x)|^2}{(|x|^2+\epsilon)^k}\,dx-2k\int_{\real^N}\frac{v(x)}{(|x|^2+\epsilon)^{k+1}}x\cdot\nabla v(x)\,dx\\
&=\int_{\real^N}\frac{|\nabla v(x)|^2}{(|x|^2+\epsilon)^k}\,dx+k\int_{\real^N}v(x)^2{\rm div}\left(\frac{x}{(|x|^2+\epsilon)^{k+1}}\right)\,dx\\
&=\int_{\real^N}\frac{|\nabla v(x)|^2}{(|x|^2+\epsilon)^k}\,dx+k\int_{\real^N}[N(|x|^2+\epsilon)-2(k+1)|x|^2]\frac{v(x)^2}{(|x|^2+\epsilon)^{k+2}}.
\end{split}
\end{equation*}
We thus derived the following identity
\begin{equation}\label{interm12}
\begin{split}
\int_{\real^N}&\frac{|\nabla v(x)|^2}{(|x|^2+\epsilon)^k}\,dx+\frac{1}{m-1}\int_{\real^N}\frac{v^{(m+1)/m}(x)}{(|x|^2+\epsilon)^k}\,dx\\
&+k\int_{\real^N}\frac{(N-2-2k)|x|^2+N\epsilon}{|x|^2+\epsilon}\frac{v(x)^2}{(|x|^2+\epsilon)^{k+1}}\,dx
=\int_{\real^N}\frac{|x|^{\sigma}v(x)^2}{(|x|^2+\epsilon)^k}\,dx.
\end{split}
\end{equation}
Let us next pick $k\in \big(0,(N-2)/2\big)$ such that $\mathcal{N}_{\sigma-2k}(v)<\infty$. Notice that this last property is at least satisfied by all $k\in \big(0,(N-2)/2\big)\cap \big(0,(\sigma+2)/2\big)$. Then, on the one hand,
$$
\frac{(N-2-2k)|x|^2+N\epsilon}{|x|^2+\epsilon}\geq N-2-2k.
$$
On the other hand,
\begin{equation*}
	\frac{|x|^{\sigma}}{(|x|^2+\epsilon)^k} \le |x|^{\sigma-2k}, \quad x\in\mathbb{R}^N.
\end{equation*}
Inserting the above estimates into~\eqref{interm12}, we further deduce that
\begin{equation*}
\int_{\real^N} \frac{|\nabla v(x)|^2}{(|x|^2+\epsilon)^k}\,dx  +k(N-2-2k)\int_{\real^N}\frac{v(x)^2}{(|x|^2+\epsilon)^{k+1}}\,dx \leq \mathcal{N}_{\sigma-2k}^2(v).
\end{equation*}
Letting $\epsilon\to 0$ in the above inequality and applying Fatou's lemma give
\begin{equation*}
\mathcal{N}_{-2k}^2(|\nabla v|) + k(N-2-2k) \mathcal{N}_{-2k-2}^2(v) \leq \mathcal{N}_{\sigma-2k}^2(v)<\infty.
\end{equation*}
Taking $k=-\big(\tau+2)/2$, we have thus established that
\begin{equation}
	\tau\in (-N,-2) \;\text{ and }\; \mathcal{N}_{\tau+\sigma+2}(v)< \infty \;\Longrightarrow\; \mathcal{N}_\tau(v) < \infty \;\text{ and }\; \mathcal{N}_{\tau+2}(|\nabla v|) < \infty. \label{interm13}
\end{equation}

We now consider $\tau_0\in (-N,-2)$ such that $\tau_0+2 \not\in (\sigma+2)\mathbb{Z}$ and let $j_0\in \mathbb{N}$ be the smallest integer such that $j_0(\sigma+2)>-(\tau_0+2)$. Owing to the negativity of $\tau_0+2$ and $\sigma$, the positivity of $\sigma+2$, and the assumptions on $\tau_0$ and $j_0$, one has
\begin{equation}
	0 < - \frac{\tau_0+2}{\sigma+2} < j_0 < 1 - \frac{\tau_0+2}{\sigma+2} = - \frac{\tau_0-\sigma}{\sigma+2} < - \frac{\tau_0}{\sigma+2}. \label{add01}
\end{equation}
Introducing
\begin{equation*}
	\tau_j := \tau_0 + j (\sigma+2), \quad 1 \le j \le j_0,
\end{equation*}
we infer from~\eqref{add01} and the positivity of $\sigma+2$ that
\begin{equation}
	- N < \tau_0 \le \tau_j \le \tau_{j_0-1} < - 2 < \tau_{j_0}<0, \quad 0\le j \le j_0-1. \label{add02}
\end{equation}
Since $\tau_j\in (-N,-2)$ for $0\le j \le j_0-1$ by~\eqref{add02}, we deduce from~\eqref{interm13} that
\begin{equation}
	\begin{split}
	& \mathcal{N}_{\tau_{j+1}}(v) = \mathcal{N}_{\tau_j+\sigma+2}(v)< \infty\\
	& \qquad  \;\Longrightarrow\; \mathcal{N}_{\tau_j}(v) < \infty \;\text{ and }\; \mathcal{N}_{\tau_j+2}(|\nabla v|) < \infty, \quad 0\le j \le j_0-1. 
	\end{split}\label{add03}
\end{equation}
Since $\mathcal{N}_{\tau_{j_0}}(v)<\infty$ by~\eqref{add02} and Lemma~\ref{lem.embed}, Lemma~\ref{lem.moments} readily follows from~\eqref{add03} for any $\tau_0\in (-N,-2)$ with $\tau_0+2 \not\in (\sigma+2)\mathbb{Z}$. The extension to any $\tau_0\in (-N,0)$ is then done by interpolation.
\end{proof}

We next prove that $v\in L^p(\real^N)$ for any $p\in(1,\infty)$. This is the last piece of information needed in order to obtain the claimed regularity in Theorem~\ref{th.1}.

\begin{lemma}\label{lem.Lp}
Let $N\geq3$ and $v\in \mathcal{X}$ be a non-negative radially symmetric solution to~\eqref{eq1} with non-increasing profile. Then $v\in L^p(\real^N)$ for any $p\in(1,\infty)$.
\end{lemma}

\begin{proof}
For $\varrho>0$, we introduce the truncation $T_\varrho(z)=\min\{z,\varrho\}$, $z>0$. For $p>1$, we infer from~\eqref{eq1} that
$$
-\int_{\real^N}T_{\varrho}(v(x))^p \Delta v(x)\,dx + \int_{\real^N}\frac{T_{\varrho}(v(x))^p}{m-1} v^{1/m}(x)\,dx = \int_{\real^N}|x|^{\sigma} v(x) T_{\varrho}(v(x))^p\,dx,
$$
whence, thanks to the non-negativity of the second term in the left hand side,
\begin{equation}\label{interm15}
p\int_{\real^N}T_{\varrho}(v(x))^{p-1} T_{\varrho}'(v(x)) |\nabla v(x)|^2\,dx \leq \int_{\real^N}|x|^{\sigma} v(x) T_{\varrho}(v(x))^p\,dx.
\end{equation}
Taking into account that $(T_{\varrho}')^2=T_{\varrho}'$, we deduce from~\eqref{interm15} that
\begin{equation*}
	\frac{4p}{(p+1)^2}\|\nabla T_{\varrho}(v)^{(p+1)/2}\|_2^2 \leq \int_{\real^N} |x|^{\sigma} v(x) T_{\varrho}(v(x))^p\,dx.
\end{equation*}
Combining the above inequality with Sobolev's inequality~\eqref{Isobolev} and the elementary inequality $4p\ge 2(p+1)$, we obtain
\begin{equation}\label{interm16}
\frac{2}{(p+1)C_S^2} \left\|T_{\varrho}(v)^{(p+1)/2}\right\|_{2^*}^2 \leq \int_{\real^N} |x|^{\sigma} v(x) T_{\varrho}(v(x))^p\,dx.
\end{equation}
We now have to estimate the right hand side of~\eqref{interm16} and argue differently for $N\in\{3,4\}$ and $N\ge 5$.

\noindent\textbf{Case~1: $N\in\{3,4\}$.} We set
\begin{equation*}
	\theta := \frac{N(p+1)(p-1)}{p(Np+4-N)} \in (0,1),
\end{equation*}
since $p>1$, $N\in\{3,4\}$, and
\begin{equation*}
	1 - \theta = \frac{p(4-N)+N}{p(Np+4-N)}>0.
\end{equation*}
As $T_\varrho(v)\le v$,
\begin{align*}
	& \int_{\real^N} |x|^{\sigma} v(x) T_{\varrho}(v(x))^p\,dx \\
	& \qquad = \int_{\real^N} |x|^{\sigma} v(x) T_{\varrho}(v(x))^{p(1-\theta)} T_{\varrho}(v(x))^{p\theta}\,dx \\
	& \qquad \le \int_{\real^N} |x|^{\sigma} v(x)^{1+p(1-\theta)} \left[ T_{\varrho}(v(x))^{(p+1)/2} \right]^{2p\theta/(p+1)}\,dx \\
	& \qquad = \int_{\real^N} |x|^{\sigma} v(x)^{4(p+1)/(Np+4-N)} \left[ T_{\varrho}(v(x))^{(p+1)/2} \right]^{2N(p-1)/(Np+4-N)}\,dx.
\end{align*}
Since
\begin{equation*}
	\frac{2(p+1)}{Np+4-N} + \frac{(N-2)(p-1)}{Np+4-N} = 1,
\end{equation*}
we may apply H\"older's inequality to obtain
\begin{align*}
	 \int_{\real^N} |x|^{\sigma} v(x) T_{\varrho}(v(x))^p\,dx & \le \mathcal{N}_\kappa(v)^{4(p+1)/(Np+4-N)} \\
	 & \hspace{2cm} \times \left\| T_{\varrho}(v)^{(p+1)/2} \right\|_{2^*}^{2^*(N-2)(p-1)/(Np+4-N)},
\end{align*}
with $\kappa := \sigma(Np+4-N)/[2(p+1)]<0$. Note that the constraints $\sigma>-2>-N$ and $p>1$ guarantee that
\begin{equation*}
	\kappa+N = \frac{N(p-1)(\sigma+2) + 4(\sigma+N)}{2(p+1)} > 0,
\end{equation*}
so that $\mathcal{N}_\kappa(v)<\infty$ by Lemma~\ref{lem.moments}.

Combining~\eqref{interm16} and the above inequality gives
\begin{equation*}
	\left\|T_{\varrho}(v)^{(p+1)/2}\right\|_{2^*}^2 \le C(p,v) \left\| T_{\varrho}(v)^{(p+1)/2} \right\|_{2^*}^{2^*(N-2)(p-1)/(Np+4-N)},
\end{equation*}
whence
\begin{equation*}
	\|T_\varrho(v)\|_{N(p+1)/(N-2)}^{4(p+1)/(Np+4-N)} = \left\|T_{\varrho}(v)^{(p+1)/2}\right\|_{2^*}^{8/(Np+4-N)} \le C(p,v).
\end{equation*}
Since the right hand side of the above inequality does not depend on $\varrho>0$, we may let $\varrho\to\infty$ in the above inequality and use Fatou's lemma to conclude that $v\in L^{N(p+1)/(N-2)}(\mathbb{R}^N)$. Since $p>1$ is arbitrary, we have thus established that $v\in L^p(\mathbb{R}^N)$ for all $p\in (2^*,\infty)$, which implies, together with the compactness of the support of $v$ proved in Proposition~\ref{prop.compact}, that $v\in L^p(\mathbb{R}^N)$ for all $p\in [1,\infty)$.

\noindent\textbf{Case~2: $N\ge 5$.} In higher space dimensions, it does not seem possible to derive directly the $L^p$-integrability of $v$ for arbitrary values of $p>1$ and we instead use an iterative argument. As a first step, we consider $p>\max\{1,2/(N-4)\}$ such that $v\in L^{(N-4)(p+1)/(N-2)}(\mathbb{R}^N)$. By H\"older's inequality,
\begin{align*}
	\int_{\real^N} & |x|^{\sigma} v(x) T_{\varrho}(v(x))^p\,dx \\
	& = \int_{\real^N} |x|^{\sigma} v(x)^{4/N} v(x)^{(N-4)/N} \left[ T_{\varrho}(v(x))^{(p+1)/2} \right]^{2p/(p+1)}\,dx\\
	& \leq \mathcal{N}_{\sigma N/2}(v)^{4/N} \|v\|_{(N-4)(p+1)/(N-2)}^{(N-4)/N} \left\| T_{\varrho}(v)^{(p+1)/2}\right\|_{2^*}^{2p/(p+1)},
\end{align*}
since
\begin{equation*}
	\frac{2}{N} + \frac{N-2}{N(p+1)} + \frac{2p}{2^*(p+1)} = 1.
\end{equation*}
Observing that $\sigma N/2>-N$, it follows from Lemma~\ref{lem.moments} that $\mathcal{N}_{\sigma N/2}(v)$ is finite and we deduce from the above inequality that
\begin{equation*}
	\int_{\real^N} |x|^{\sigma} v(x) T_{\varrho}(v(x))^p\,dx \le C_0(v) \|v\|_{(N-4)(p+1)/(N-2)}^{(N-4)/N} \left\| T_{\varrho}(v)^{(p+1)/2}\right\|_{2^*}^{2p/(p+1)},
\end{equation*}
with $C_0(v):= \mathcal{N}_{\sigma N/2}(v)^{4/N}$. We then infer from~\eqref{interm16} and the above inequality that
\begin{equation*}
	\left\|T_{\varrho}(v)^{(p+1)/2}\right\|_{2^*}^2 \leq \frac{(p+1) C_S^2}{2} C_0(v) \|v\|_{(N-4)(p+1)/(N-2)}^{(N-4)/N} \left\| T_{\varrho}(v)^{(p+1)/2}\right\|_{2^*}^{2p/(p+1)},
\end{equation*}
from which we deduce that
\begin{equation*}
	\|T_{\varrho}(v)\|_{N(p+1)/(N-2)} = \left\|T_{\varrho}(v)^{(p+1)/2}\right\|_{2^*}^{2/(p+1)} \le (p+1) C_S^2 C_0(v) \|v\|_{(N-4)(p+1)/(N-2)}^{(N-4)/N}.
\end{equation*}
Since the right hand side of the above inequality is assumed to be finite and does not depend on $\varrho>0$, we may let $\varrho\to\infty$ and use Fatou's lemma to conclude that $v\in L^{N(p+1)/(N-2)}(\mathbb{R}^N)$ with
\begin{equation*}
	\|v\|_{N(p+1)/(N-2)} \le (p+1) C_S^2 C_0(v) \|v\|_{(N-4)(p+1)/(N-2)}^{(N-4)/N}.
\end{equation*}
Setting $q=(N-4)(p+1)/(N-2)$, we have shown that, for $q>\max\{1,2(N-4)/(N-2)\}$,
\begin{equation}
	v\in L^q(\mathbb{R}^N) \Longrightarrow \left\{
	\begin{array}{l}
		v\in L^{Nq/(N-4)}(\mathbb{R}^N) \\
		\;\text{ with }\; \|v\|_{Nq/(N-4)} \le \frac{q(N-2)}{N-4} C_S^2 C_0(v) \|v\|_{q}^{(N-4)/N}
	\end{array} \right. \label{add04}
\end{equation}
Introducing the sequence
\begin{equation*}
	q_j = 2 \left( \frac{N}{N-4} \right)^j, \quad j\ge 0,
\end{equation*}
and recalling that $v\in L^2(\mathbb{R}^N)$, we infer from~\eqref{add04} by induction that $v\in L^{q_j}(\mathbb{R}^N)$ for all $j\ge 0$. Since $q_j\to\infty$ as $j\to\infty$, an interpolation argument entails that $v\in L^p(\mathbb{R}^N)$ for all $p\in [2,\infty)$, while the $L^p$-integrability of $v$ for $p\in [1,2)$ follows from its $L^2$-integrability and the compactness of its support.
\end{proof}

Notice that we cannot yet deduce at this stage that $v\in L^{\infty}(\real^N)$ for $N\ge 2$. Indeed, the estimates used in the proof of Lemma~\ref{lem.Lp} feature an unbounded dependence on $p$ when $N\ge 3$ and the Sobolev embedding only provides that $v\in L^p(\real^N)$ for any $p\in(1,\infty)$ when $N=2$. Thus, in order to obtain both the boundedness for $N\geq2$ and the regularity claimed in Theorem~\ref{th.2}, we still need to perform one more step, as indicated below.

\begin{proof}[Proof of Theorem~\ref{th.2}: regularity]
Let $N\ge 1$ and let $v\in \mathcal{X}$ be a non-negative radially symmetric solution to~\eqref{eq1} with non-increasing profile. For $p\in [1,\infty)$,  we have $v\in L^p(\real^N)$ by Lemma~\ref{lem.Lp} for $N\ge 3$ and by Sobolev embeddings for $N\in\{1,2\}$. Therefore, for $x\in\mathbb{R}^N$,
\begin{equation*}
\omega_N|x|^Nv(x)\leq\int_{B(0,|x|)}v(x)\,dy\leq\int_{B(0,|x|)}v(y)\,dy\leq\|v\|_{p}(\omega_N|x|^{N})^{(p-1)/p}.
\end{equation*}
We thus infer that
\begin{equation}\label{interm19}
|x|^{\sigma}v(x)\leq\omega_N^{-1/p}|x|^{(\sigma p-N)/p}\|v\|_p, \quad x\in\mathbb{R}^N.
\end{equation}
Fix now $q\in(N/2,N/|\sigma|)$, which is possible since $\sigma\in(-2,0)$. Then, since $v$ is compactly supported inside a ball $B(0,R_0)$ by Proposition~\ref{prop.compact}, we deduce from~\eqref{interm19} that
\begin{align*}
\int_{\real^N}(|x|^{\sigma}v(x))^q\,dx & \leq\omega_N^{-q/p}\int_{B(0,R_0)}|x|^{q(\sigma p -N)/p}\|v\|_p^{q}\,dx \\
& \leq C(p,v)\int_0^{R_0}r^{q(\sigma p-N)/p+N-1}\,dr.
\end{align*}
We now choose $p>Nq/(N+q\sigma)>1$, so that the right hand side of the above estimate is finite. Therefore, $x\mapsto|x|^{\sigma}v(x)$ belongs to $L^q(\real^N)$ for any $q\in(N/2,N/|\sigma|)$. Moreover, since also $v^{1/m}\in L^q(\real^N)$, we deduce from~\eqref{eq1} that $-\Delta v\in L^q(\real^N)$, which, together with the compactness of the support of $v$ and standard elliptic regularity, ensure that $v\in W^{2,q}(\real^N)$. As $q$ is arbitrary in the interval $(N/2,N/|\sigma|)$, the continuous embedding of $W^{2,q}(\real^N)$ into $L^{\infty}(\real^N)$ implies that $v\in L^{\infty}(\real^N)$. In fact, we can even extend this regularity to $v\in C^{\alpha}(\real^N)$ for $\alpha=2-N/q$ if $\sigma\in(-2,-1]$ and to $v\in C^{1,\alpha}(\real^N)$ for $\alpha=1-N/q$ provided $\sigma\in(-1,0)$ (and choosing $q\in(N,N/|\sigma|)$).
\end{proof}

\subsection{Expansion as $|x|\to 0$}\label{subsec.ez}

The aim of this section is to compute the local expansion of $v$ in a neighborhood of $x=0$ and thus complete the proof of Theorem~\ref{th.2}. Recalling the notation $V(r)=V(|x|)=v(x)$, $r=|x|$, for the profile of $v$, and owing to the embedding of $W^{2,q}(\real^N)$ in $BC(\real^N)$ for $q\in(\max\{1,N/2\},N/|\sigma|)$, we infer that $v(0)=V(0)$ is well defined. Since $v$ is radially symmetric, \eqref{eq1} becomes an ordinary differential equation for $V$ with respect to the radial variable $r=|x|$; that is,
\begin{equation}\label{eq1.radial}
	\big(r^{N-1}V'(r)\big)'=\frac{r^{N-1}}{m-1}V^{1/m}(r)-r^{N-1+\sigma}V(r), \quad r\in (0,\infty).
\end{equation}
We now proceed with the final step of the proof of Theorem~\ref{th.2}. 

\begin{proof}[Proof of Theorem~\ref{th.2}: expansion as $|x|\to0$]
Since $\sigma<0$ and $V$ is bounded, we deduce from \eqref{eq1.radial} that, on the one hand,
\begin{equation}\label{interm20}
	(r^{N-1}V'(r))'\sim-r^{N+\sigma-1}V(r) \;\;\text{ as }\;\; r\to 0,
\end{equation}
while, on the other hand, the radial symmetry and the condition $\sigma+N>0$ entail
$$
\lim\limits_{r\to0}r^{N-1}V'(r)=\lim\limits_{r\to0}\frac{r^{N+\sigma}}{N+\sigma}=0.
$$
We infer by combining the previous limit with~\eqref{interm20} and the L'Hospital rule that
$$
\lim\limits_{r\to 0} \frac{(N+\sigma)r^{N-1}V'(r)}{r^{N+\sigma}}=-V(0),
$$
or equivalently,
\begin{equation}\label{interm21}
	V'(r)\sim-\frac{V(0)}{N+\sigma}r^{\sigma+1} \quad {\rm as} \ r\to0.
\end{equation}
Taking into account that $\sigma\in(-2,0)$, we next get by integration from~\eqref{interm21} that
\begin{equation}\label{interm22}
	V(r) = V(0)-\frac{V(0)}{(N+\sigma)(\sigma+2)}r^{\sigma+2} +o(r^{\sigma+2}).
\end{equation}
We next go to the second order derivative and we deduce from~\eqref{interm20}, \eqref{interm21}, and~\eqref{interm22} that 
\begin{equation*}
	\begin{split}
		r^{N-1}V''(r)&=-(N-1)r^{N-2}V'(r)+(r^{N-1}V'(r))'\\
		&=-(N-1)r^{N-2}V'(r) - r^{N-1+\sigma}V(r) + o(r^{N-1+\sigma})\\
		&=\frac{N-1}{N+\sigma}V(0)r^{N+\sigma-1}-V(0)r^{N+\sigma-1}+o(r^{N+\sigma-1}),
	\end{split}
\end{equation*}
whence
\begin{equation}\label{interm23}
	V''(r)=-\frac{\sigma+1}{N+\sigma}V(0)r^{\sigma}+o(r^{\sigma}).
\end{equation}
At this point, let us remark that the expansion~\eqref{interm23} provides no information on $V''$ when $\sigma=-1$. Thus, we are led to study the next order of the expansion. We thus replace the first order expansion~\eqref{interm22} into~\eqref{eq1.radial} and obtain
\begin{equation*}
	\begin{split}
		(r^{N-1}V'(r))'&=\frac{r^{N-1}}{m-1}V(0)^{1/m}\left[1-\frac{r^{\sigma+2}}{m(N+\sigma)(\sigma+2)}+o(r^{\sigma+2})\right]\\
		&-r^{N+\sigma-1}\left[V(0)-\frac{V(0)}{(N+\sigma)(\sigma+2)}r^{\sigma+2}+o(r^{\sigma+2})\right]\\
		&=-V(0)r^{N+\sigma-1}+\frac{V(0)^{1/m}}{m-1}r^{N-1}+o(r^{N-1})\\
		&+\frac{V(0)}{(N+\sigma)(\sigma+2)}r^{N+2\sigma+1}+o(r^{N+2\sigma+1}).
	\end{split}
\end{equation*}
We next split the rest of our analysis into the three cases mentioned in the statement of Theorem \ref{th.2}.

\medskip

\noindent \textbf{Case 1: $\sigma\in(-1,0)$.} We observe that $N+2\sigma+1>N-1$, whence
\begin{equation}\label{interm25}
	(r^{N-1}V'(r))'=-V(0)r^{N+\sigma-1}+\frac{V(0)^{1/m}}{m-1}r^{N-1}+o(r^{N-1}),
\end{equation}
and by integration we further infer that
\begin{equation}\label{interm24}
	V'(r)=-\frac{V(0)}{N+\sigma}r^{\sigma+1}+\frac{V(0)^{1/m}}{N(m-1)}r+o(r).
\end{equation}
Since
\begin{equation}\label{eq1.second}
	V''(r)=-\frac{N-1}{r}V'(r)+\frac{1}{r^{N-1}}(r^{N-1}V'(r))',
\end{equation}
we readily obtain the expansion~\eqref{V0_1} by replacing the terms in the right hand side of~\eqref{eq1.second} by their expansions in~\eqref{interm25} and~\eqref{interm24}.

\medskip

\noindent \textbf{Case 2: $\sigma\in(-2,-1)$.} We observe that $N+2\sigma+1<N-1$, whence
\begin{equation}\label{interm26}
	(r^{N-1}V'(r))'=-V(0)r^{N+\sigma-1}+\frac{V(0)}{(N+\sigma)(\sigma+2)}r^{N+2\sigma+1}+o(r^{N+2\sigma+1}),
\end{equation}
and by integration, taking into account that $N+2\sigma+2=N+\sigma+\sigma+2>0$, we further infer that
\begin{equation}\label{interm27}
	V'(r)=-\frac{V(0)}{N+\sigma}r^{\sigma+1}+\frac{V(0)}{(N+\sigma)(\sigma+2)(N+2\sigma+2)}r^{2\sigma+3}+o(r^{2\sigma+3}).
\end{equation}
We then readily obtain the expansion~\eqref{V0_3} by replacing the terms in the right hand side of \eqref{eq1.second} by their expansions in~\eqref{interm26} and~\eqref{interm27} and performing some straightforward calculations.

\medskip

\noindent \textbf{Case 3: $\sigma=-1$.} In this case, $N+2\sigma+1=N-1$, whence
\begin{equation}\label{interm28}
	(r^{N-1}V'(r))'=-V(0)r^{N-2}+\left[\frac{V(0)}{N-1}+\frac{V(0)^{1/m}}{m-1}\right]r^{N-1}+o(r^{N-1}),
\end{equation}
from which we deduce by integration that
\begin{equation}\label{interm29}
	V'(r)=-\frac{V(0)}{N-1}+\left[\frac{V(0)}{N-1}+\frac{V(0)^{1/m}}{m-1}\right]\frac{r}{N}+o(r).
\end{equation}
We again obtain the expansion~\eqref{V0_2} by replacing the terms in the right hand side of~\eqref{eq1.second} by their expansions in~\eqref{interm28} and~\eqref{interm29}. The proof of Theorem~\ref{th.2} is now complete.
\end{proof}

\begin{remark}\label{rem.reg}
	As already pointed out in the introduction, the expansions~\eqref{V0_1} and~\eqref{V0_3}, along with the negativity of $\sigma$, show that we cannot expect the profile $V$ to be $C^2$-smooth at $r=0$. Furthermore, if $N\geq2$ and $\sigma\in(-2,-1)$, then even the $C^1$-regularity of $V$ fails to be true. However, due to the fact that the first order of the expansion vanishes for $\sigma=-1$, we surprisingly have $V\in C^2([0,\infty))$ in this case, as shown by the expansion~\eqref{V0_2}. This discussion concludes the sharpness of Theorem~\ref{th.2} with respect to the regularity properties of $V$.
\end{remark}

\section{Non-existence}\label{sec.ne}

This section is dedicated to the proof of Theore~ \ref{th.3} by means of a Pohozaev identity that will be deduced below following analogous steps as in \cite{BL83, FT00}.
\begin{proof}[Proof of Theorem~\ref{th.3}]
	As usual, we multiply equation~\eqref{eq1} by $x\cdot\nabla v$ and integrate by parts. We thus have
	\begin{equation*}
		\begin{split}
			-\int_{\real^N}(x\cdot\nabla v)\Delta v\,dx&=\int_{\real^N}\nabla v\cdot\nabla(x\cdot \nabla v)\,dx
			=\int_{\real^N}\sum\limits_{i=1}^{N}\partial_iv\partial_i\left( \sum\limits_{j=1}^Nx_j\partial_jv \right)\,dx\\
			&=\sum\limits_{i,j=1}^N\int_{\real^N}(x_j\partial_iv\partial_j\partial_iv+\delta_{ij}\partial_iv\partial_jv)\,dx\\
			&=\frac{1}{2}\int_{\real^N}x\cdot\nabla(|\nabla v|^2)\,dx+\|\nabla v\|_{2}^2\\
			&=-\frac{N}{2}\|\nabla v\|_2^2+\|\nabla v\|_2^2=-\frac{N-2}{2}\|\nabla v\|_2^2.
		\end{split}
	\end{equation*}
	Similarly,
	$$
	\int_{\real^N}(x\cdot\nabla v)|x|^{\sigma}v\,dx=\frac{1}{2}\int_{\real^N}|x|^{\sigma}x\cdot\nabla v^2\,dx=-\frac{N+\sigma}{2} \mathcal{N}_{\sigma}^2(v),
	$$
	and
	$$
	\int_{\real^N}(x\cdot\nabla v)v^{1/m}\,dx=\frac{m}{m+1}\int_{\real^N}x\cdot\nabla v^{(m+1)/m}\,dx=-\frac{mN}{m+1}\int_{\real^N}v^{(m+1)/m}\,dx.
	$$
	Gathering the previous identities, we infer from~\eqref{eq1} that
	\begin{equation}\label{poh1}
		-\frac{N-2}{2}\|\nabla v\|_{2}^2-\frac{mN}{(m+1)(m-1)}\|v\|_{(m+1)/m}^{(m+1)/m}=-\frac{N+\sigma}{2} \mathcal{N}_{\sigma}^2(v).
	\end{equation}
	By multiplying equation~\eqref{eq1} by $v$ and integrating over $\real^N$, we derive the following identity
	\begin{equation}\label{poh2}
		\|\nabla v\|_{2}^2+\frac{1}{m-1}\|v\|_{(m+1)/m}^{(m+1)/m}= \mathcal{N}_{\sigma}^2(v).
	\end{equation}
	We proceed as in \cite[Section~5]{FT00} and we multiply~\eqref{poh2} by $(N+\sigma)/2$ and add the resulting equality to~\eqref{poh1} in order to eliminate the terms in the right hand side. We thus obtain
	\begin{equation}\label{poh3}
		\frac{\sigma+2}{2}\|\nabla v\|_2^{2}+\frac{\sigma(m+1)-N(m-1)}{2(m^2-1)}\|v\|_{(m+1)/m}^{(m+1)/m}=0.
	\end{equation}
	We next observe that, as $\sigma\leq-2$, the coefficient of the first term in~\eqref{poh3} is non-positive, while the coefficient of the second term in~\eqref{poh3} is negative. This readily implies that $\|v\|_{(m+1)/m}=0$, hence $v\equiv0$.
\end{proof}

\section*{Acknowledgements} This work is partially supported by the Spanish project PID2020-115273GB-I00 and by the Grant RED2022-134301-T (Spain). Part of this work has been developed during visits of R. G. I. to Laboratoire de Math\'ematiques LAMA, Universit\'e de Savoie, and of Ph. L. to Universidad Rey Juan Carlos, and both authors thank these institutions for hospitality and support.

\bibliographystyle{plain}

\end{document}